\documentclass[10pt,reqno]{amsproc}

\usepackage{amssymb}
\usepackage{bbm}
\usepackage{amsthm}
\usepackage[bookmarks=false]{hyperref}
\usepackage{etoolbox}
\usepackage{array}
\usepackage{xparse}
\usepackage{enumitem}
\usepackage{tikz}
\usepackage{mathtools}
\usepackage[font=scriptsize]{caption}

\DeclareMathOperator{\cnx}{div}
\DeclareMathOperator{\cn}{div}

\DeclareMathOperator{\dif}{d}

\def\ba{\begin{align}}
\def\bad{\begin{aligned}}
\def\be{\begin{equation}}
\def\ea{\end{align}}
\def\ead{\end{aligned}}
\def\ee{\end{equation}}
\def\e{\eqref}

\def\dsigma{\dif \! \sigma}
\def\dx{\dif \! x}
\def\dalpha{\dif \! \alpha}

\def\dydx{\dif \! y \dif \! x}

\def\defn{\mathrel{:=}}

\def\eps{\varepsilon}
\def\la{\left\vert}
\def\lA{\left\Vert}
\def\le{\leq}

\def\mez{\frac{1}{2}}

\def\ra{\right\vert}
\def\rA{\right\Vert}

\def\xN{\mathbf{N}}
\def\xR{\mathbf{R}}

\def\xT{\mathbf{T}}

\newtheorem{thm}{Theorem}[section]
\newtheorem{theorem}[thm]{Theorem}

\newtheorem{prop}[thm]{Proposition}
\newtheorem{cor}[thm]{Corollary}

\theoremstyle{definition}

\theoremstyle{remark}
\newtheorem{rem}[thm]{Remark}
\newtheorem{rmk}[thm]{Remark}

\newcommand{\thmref}[1]{Theorem~\ref{#1}}

\newcommand{\propref}[1]{Proposition~\ref{#1}}

\newcommand*{\qq}{\qquad}

\newcommand*{\tx}[1]{\text{#1}}

\newcommand*{\suchthat}{\, \middle| \,}

\newcommand*{\myoverline}[3]{\mkern -#1mu\overline{\mkern#1mu#3\mkern#2mu}\mkern -#2mu}	

\newcommand*{\Pminusbar}{\myoverline{-3}{0}{P}_{-}}

\newcommand*{\half}{\frac{1}{2}}

\newcommand*{\Zsp}{\mathbb{Z}}

\newcommand*{\Pminus}{P_{-}}

\newcommand*{\Ltwo}{L^2}

\newcommand*{\al}{\alpha}

\newcommand*{\diff}{\mathop{}\! d}

\newcommand*{\Hil}{\mathbb{H}}

\newcommand*{\Imag}{\tx{Im}}
\newcommand*{\Real}{\tx{Re}}

\newcommand*{\grad}{\nabla}

\newcommand*{\Dabs}{\abs{D}}

\newcommand*{\pal}{\partial_\al}

\newcommand*{\Fcal}{\mathcal{F}}

\newcommand*{\zetatil}{\widetilde{\zeta}}

\newcommand*{\Z}{Z}

\newcommand*{\Zal}{\Z_{\al}}

\newcommand*{\Zalabs}{\abs{\Zal}}

\DeclarePairedDelimiter{\oldbrac}{\lparen}{\rparen}			
\NewDocumentCommand{\brac}{ s o m }{						
	\IfBooleanT{#1}{
  		\IfValueT{#2}{\oldbrac[#2]{#3}}
		\IfValueF{#2}{\oldbrac{#3}} 
	}
	\IfBooleanF{#1}{
  		\IfValueT{#2}{\PackageError{mypackage}{Incorrect use of brac. Insert star}{}}
		\IfValueF{#2}{\oldbrac*{#3}} 
	}		
}

\DeclarePairedDelimiter\oldcbrac{\lbrace}{\rbrace}				
\NewDocumentCommand{\cbrac}{ s o m }{					
	\IfBooleanT{#1}{
  		\IfValueT{#2}{\oldcbrac[#2]{#3}}
		\IfValueF{#2}{\oldcbrac{#3}} 
	}
	\IfBooleanF{#1}{
  		\IfValueT{#2}{\PackageError{mypackage}{Incorrect use of cbrac. Insert star}{}}
		\IfValueF{#2}{\oldcbrac*{#3}} 
	}		
}

\DeclarePairedDelimiter\oldsqbrac{\lbrack}{\rbrack}				
\NewDocumentCommand{\sqbrac}{ s o m }{					
	\IfBooleanT{#1}{
  		\IfValueT{#2}{\oldsqbrac[#2]{#3}}
		\IfValueF{#2}{\oldsqbrac{#3}} 
	}
	\IfBooleanF{#1}{
  		\IfValueT{#2}{\PackageError{mypackage}{Incorrect use of sqbrac. Insert star}{}}
		\IfValueF{#2}{\oldsqbrac*{#3}} 
	}		
}

\DeclarePairedDelimiter{\oldabs}{\lvert}{\rvert}
\NewDocumentCommand{\abs}{ s o m }{						
	\IfBooleanT{#1}{
  		\IfValueT{#2}{\oldabs[#2]{#3}}
		\IfValueF{#2}{\oldabs{#3}} 
	}
	\IfBooleanF{#1}{
  		\IfValueT{#2}{\PackageError{mypackage}{Incorrect use of abs. Insert star}{}}
		\IfValueF{#2}{\oldabs*{#3}} 
	}		
}

\DeclarePairedDelimiterX{\oldnorm}[1]{\lVert}{\rVert}{#1}
\NewDocumentCommand{\norm}{ s o o m }{					
	\IfValueT{#2} {
		\IfBooleanT{#1}{
  			\IfValueT{#3}{\oldnorm[#2]{#4}_{#3}}
			\IfValueF{#3}{\oldnorm{#4}_{#2}} 
		}
		\IfBooleanF{#1}{
  			\IfValueT{#3}{\PackageError{mypackage}{Incorrect use of norm. Insert star}{}}
			\IfValueF{#3}{\oldnorm*{#4}_{#2}} 
		}
	}
	\IfValueF{#2} {
		\IfBooleanT{#1}{\oldnorm{#4}}	
		\IfBooleanF{#1}{\oldnorm*{#4}}		
	}	
}

\makeatletter
\def\black@#1{%
    \noalign{%
        \ifdim#1>\displaywidth
            \dimen@\prevdepth
            \nointerlineskip
            \vskip-\ht\strutbox@
            \vskip-\dp\strutbox@
            \vbox{\noindent\hbox to \displaywidth{\hbox to#1{\strut@\hfill}}}%
            \prevdepth\dimen@
        \fi
    }%
}
\makeatother

\makeatletter
\renewcommand{\tocsection}[3]{%
  \indentlabel{\@ifnotempty{#2}{\bfseries\ignorespaces#1 #2\quad}}\bfseries#3}
\renewcommand{\tocsubsection}[3]{%
  \indentlabel{\@ifnotempty{#2}{\ignorespaces#1 #2\quad}}#3}

\newcommand\@dotsep{4.5}
\def\@tocline#1#2#3#4#5#6#7{\relax
  \ifnum #1>\c@tocdepth 
  \else
    \par \addpenalty\@secpenalty\addvspace{#2}%
    \begingroup \hyphenpenalty\@M
    \@ifempty{#4}{%
      \@tempdima\csname r@tocindent\number#1\endcsname\relax
    }{%
      \@tempdima#4\relax
    }%
    \parindent\z@ \leftskip#3\relax \advance\leftskip\@tempdima\relax
    \rightskip\@pnumwidth plus1em \parfillskip-\@pnumwidth
    #5\leavevmode\hskip-\@tempdima{#6}\nobreak
    \leaders\hbox{$\m@th\mkern \@dotsep mu\hbox{.}\mkern \@dotsep mu$}\hfill
    \nobreak
    \hbox to\@pnumwidth{\@tocpagenum{\ifnum#1=1\bfseries\fi#7}}\par
    \nobreak
    \endgroup
  \fi}
\AtBeginDocument{%
\expandafter\renewcommand\csname r@tocindent0\endcsname{0pt}
}
\def\l@subsection{\@tocline{2}{0pt}{2.5pc}{5pc}{}}
\makeatother

\makeatletter
 \def\@testdef #1#2#3{%
   \def\reserved@a{#3}\expandafter \ifx \csname #1@#2\endcsname
  \reserved@a  \else
 \typeout{^^Jlabel #2 changed:^^J%
 \meaning\reserved@a^^J%
 \expandafter\meaning\csname #1@#2\endcsname^^J}%
 \@tempswatrue \fi}
\makeatother

\makeatletter
\newcommand*{\rom}[1]{\expandafter\@slowromancap\romannumeral #1@}
\makeatother

\makeatletter
\patchcmd{\@sect}{\@addpunct.}{}{}{}
\patchcmd{\subsection}{-.5em}{1em}{}{}
\makeatother

\newcommand*{\Bcal}{\mathcal{B}}
\newcommand*{\Vcal}{\mathcal{V}}

\numberwithin{equation}{section}
\begin{document}

\title[Rellich]{Refined Rellich boundary inequalities for the derivatives of a harmonic function}
\author{Siddhant Agrawal}
\author{Thomas Alazard}
\date{}

\begin{abstract}
The classical Rellich inequalities imply that the $L^2$-norms of the normal and tangential derivatives of a harmonic function are equivalent. In this note, we prove several refined inequalities, which make sense even if the domain is not Lipschitz. For two-dimensional domains, we obtain a sharp $L^p$-estimate for $1<p\leq 2$ by using a Riemann mapping and interpolation argument.
\end{abstract}


\maketitle

\section{Introduction}\label{S:G-Lip}

Let $d\ge 1$ and denote by $\xT^d$ a $d$-dimensional torus. 
Given two real valued functions $h\in W^{1,\infty}(\xT^d)$ and 
$\zeta\in H^{1/2}(\xT^d)$, it is classical that there exists a unique variational solution $\phi$ 
to the following problem
\begin{equation}\label{eq:phi}
\left\{
\begin{aligned}
&\Delta_{x,y}\phi=0\quad\text{in }\Omega=\{(x,y)\in \xT^d\times \xR \,;\, y<h(x)\},\\
&\phi(x,h(x)) = \zeta(x),\\
& \lim_{y\to-\infty}\sup_{x\in\xT^{d}}\la \nabla_{x,y}\phi(x,y)\ra=0.
\end{aligned}
\right.
\end{equation}
We are interested by quantitative estimates for the trace of 
the normal derivative $\partial_N\phi$ on the boundary~$\partial\Omega$, 
where the normal unit vector $N\in \xR^{d+1}$ is defined by 
\begin{align}\label{eq:normal}
N = \frac{1}{\sqrt{1+|\nabla h|^2}}\begin{pmatrix} -\nabla h\\ 1\end{pmatrix}.
\end{align}
By construction, the variational solution is such that $\nabla_{x,y}\phi \in L^2(\Omega)$, 
so it is not obvious that one can consider the trace $\partial_N \phi\arrowvert_{\partial\Omega}$. However, 
since $\Delta_{x,y}\phi=0$, one can express the normal 
derivative in terms of the tangential derivatives and prove that 
$\sqrt{1+|\nabla h|^2}\partial_N \phi\arrowvert_{\partial\Omega}$ 
is well-defined and belongs to $H^{-\mez}(\xT^d)$. 

In this paper, we are chiefly interested by another estimate, known as Rellich inequality, which plays a key role 
in the study of boundary value problems in Lipschitz domains. This inequality shows the equivalence between 
the $L^2$-norm of the tangential derivatives 
and the $L^2$-norm of the normal derivative (see~\cite{Dahlberg77,Verchota84,Dahlberg87,Gao91,Brown1994,CWGLS-2012,OttBrown2013}): there is constant $C>0$, depending only on $d$ and $\lA \nabla h\rA_{L^\infty}$ such that
\begin{equation}\label{Rellich-original}
\frac{1}{C}\int_{\partial\Omega}  (\partial_N\phi)^2 \dsigma
\le 
\int_{\xT^d}  \la\nabla \zeta\ra^2 \dx
\le C
\int_{\partial\Omega}  (\partial_N\phi)^2 \dsigma,
\end{equation}
where $\dsigma=\sqrt{1+|\nabla h|^2}\dx$ is the surface measure on $\partial\Omega$. 

The first proof of an inequality of the form \eqref{Rellich-original}
was obtained by an integration by parts argument by 
F. Rellich \cite{Re40}. He was originally interested in studying  the eigenvalues of the Laplacian in star-shaped domains. This identity plays a key role in many questions related to elliptic PDEs, for example it was used by Jerison and Kenig in their famous work on the Laplacian on Lipschitz domains \cite{JeKe80, JeKe81Dirichlet, JeKe81Neumann} and by Verchota \cite{Verchota84} who used Rellich identities implicitly in his work on layer potentials. It also plays a central role in the study of various questions in inverse problems (see e.g. \cite{AmKa04}) and acoustic scattering (see the survey paper \cite{ChGrLaSp12} which contains many references). Identities of the form obtained by Rellich also appear in many works connected to the multiplier method. The original proof of the Rellich identity makes use of the multiplier $x\cdot \nabla u$, used later by Morawetz~\cite{Morawetz1968} and J.-L. Lions~\cite{Lions1988}. Payne and Weinberger \cite{PaWe55, PaWe58} later  generalized the method and extended it to second-order elliptic systems with variable coefficients. Interestingly H\"ormander \cite{Ho54} had already obtained a general identity in 1954. In other communities, this multiplier or identity is better known as the famous Derrick-Pohozaev identity, used to prove the non-existence of solutions to some nonlinear elliptic equations. 

In this paper, we are going to prove several estimates which clarifies the dependance of the estimate~\eqref{Rellich-original} on the domain. Hereafter, given a function $f=f(x,y)$ we use $f\arrowvert_{y=h}$ as a short notation for the function $x\mapsto f(x,h(x))$.

\begin{theorem}\label{G-Lip-intro}
Let $d\ge 1$. For all $h\in C^{1}(\xT^d)$ and for all 
$\zeta\in H^1(\xT^d)$, the traces of the derivatives 
$(\nabla_{x,y}\phi)\arrowvert_{y=h}$ are well-defined and belong to $L^2(\xT^d)$. In addition, there holds
\be\label{d10-bisz-intro}
\int_{\xT^d}  (\partial_N\phi)(x,h(x))^2 \dx \le   
40\int_{\xT^d}(1+|\nabla h(x)|^2)^2 |\nabla \zeta(x)|^2 \dx,
\ee
and
\be\label{d10-bise-intro}
\int_{\xT^d} \la (\nabla_{x,y}\phi)(x,h(x))\ra^2\dx
\le   
41\int_{\xT^d}(1+|\nabla h(x)|^2)^2 |\nabla \zeta(x)|^2 \dx.
\ee
\end{theorem}
\begin{rem}\label{R:smooth}
$(i)$ Compared to \e{Rellich-original},  
the rather surprising feature of \e{d10-bisz-intro} and \e{d10-bise-intro} is the fact that the right-hand sides can be estimated even if $h$ is {\em not} a Lipschitz function. 
For example, we can write that
$$
\int_{\xT^d}(1+|\nabla h|^2)^2 |\nabla \zeta|^2 \dx\le 2\lA \nabla\zeta\rA_{L^2}^2+2\lA \nabla \zeta\rA_{L^\infty}^2\lA \nabla h\rA_{L^4}^4.
$$
In the same vein, if $\zeta=h$, we obtain from~\e{d10-bise-intro} that 
$$
\lA (\nabla_{x,y}\phi)\arrowvert_{y=h}\rA_{L^2}
\le   7 \Big( \lA \nabla h\rA_{L^2}+\lA \nabla h\rA_{L^6}^3\Big).
$$
Notice that the case $\zeta=h$ is 
interesting for the Hele-Shaw equation (see \cite{ChangLaraGuillenSchwab,AMS,NPausader,Dong-Gancedo-Nguyen}).

$(ii)$ One could extend the estimates~\e{d10-bisz-intro} and 
\e{d10-bise-intro} to the cases where 
$h$ belongs to $W^{1,\infty}(\xT^d)$ instead of $C^1(\xT^d)$ 
by using the arguments in Ne{\v c}as~\cite[Chapter 5]{Necas}, 
Brown~\cite{Brown1994} or McLean~\cite[Theorem 4.24]{McLean}. 
\end{rem}

Consider now 
the Dirichlet-to-Neumann operator $G(h)$ defined by 
$$
G(h)\zeta=\big(\partial_y\phi-\nabla h\cdot\nabla \phi\big)\big\arrowvert_{y=h}=\sqrt{1+|\nabla h|^2} \partial_N \phi\big\arrowvert_{y=h}.
$$
From the previous inequalities, we immediately obtain the following 

\begin{cor}\label{coro.G-Lip-intro}
Let $d\ge 1$. For all $h\in C^{1}(\xT^d)$ and for all 
$\zeta\in H^1(\xT^d)$, there holds 
$G(h)\zeta\in L^2(\xT^d)$ together with the estimate
\be\label{G-Lip-n0}
\int_{\xT^d} \frac{(G(h)\zeta)^2}{1+|\nabla h|^2}\dx \le   
40\int_{\xT^d}(1+|\nabla h|^2)^2 |\nabla \zeta|^2 \dx.
\ee
\end{cor}
\begin{rem}In particular,
\be\label{G-Lip-n1}
\int_{\xT^d} (G(h)\zeta)^2\dx \le   
40\,  (1+\lA \nabla h\rA_{L^\infty}^2)^3
\int_{\xT^d} |\nabla \zeta|^2 \dx.
\ee
As said above, compared to \e{G-Lip-n1}, the estimate~\e{G-Lip-n0} is 
quite surprising in that the right-hand side of the former might be finite even if $\nabla h$ is unbounded. In this case, we do not control the $L^2$-norm of $G(h)\zeta$ but only a weaker quantity. 
\end{rem}

In dimension $d=1$, we can extend the above result in two directions. The first one is a stronger version of estimate~\e{G-Lip-n0} where the right-hand side does not involve $h$ at all, while the second version generalizes to $L^p$ estimates. 
If $d=1$, we will denote simply by $f_x$ the derivative $\partial_x f$.

\begin{thm}\label{thm:Rellich1D}
For all $h\in C^{1}(\xT)$ and for all $\zeta\in H^1(\xT)$ we have
\begin{align}\label{eq:Ghzetafromzetax}
\int_{\xT} \frac{(G(h)\zeta)^2 }{1 + h_x^2} \dx \leq 4 \int_{\xT}  \zeta_x^2 \dx,
\end{align}
and
\begin{align}\label{eq:zetaxfromGhzeta}
\int_{\xT}  \frac{\zeta_x^2}{1 + h_x^2} \dx \leq 4\int_{\xT} (G(h)\zeta)^2  \dx.
\end{align}
\end{thm}

As a corollary, one can get a surprising geometric estimate. 

\begin{cor}
Denote by $\kappa$ the curvature of $\partial\Omega$ and by $\theta$ the angle the interface $\partial\Omega$ 
makes with the $x$-axis, defined by
$$
\kappa=\partial_{x}\left(\frac{h_x}{\sqrt{1+h_x^2}}\right)\quad, \quad
\theta= \arctan(h_x).
$$
Then, there holds
$$
\lA G(h)\kappa\rA_{H^{-1}}\le 2\lA \theta_x\rA_{L^2}.
$$
\end{cor}
\begin{proof}
Notice that $\kappa=h_{xx}/(1+h_x^2)^{3/2}$. 
Since $G(h)$ is self-adjoint for the $L^2$-scalar product, for any function $\varphi\in H^1(\xT)$, we deduce from \eqref {eq:Ghzetafromzetax} that
\begin{align*}
\int_{\xT}\varphi G(h)\kappa \dx
&=\int_{\xT}\kappa G(h)\varphi \dx 
\le \left(\int_{\xT} (1+h_x^2)\kappa^2\dx\right)^\mez\left(
\int_{\xT}\frac{(G(h)\varphi)^2}{1+h_x^2}\dx\right)^\mez\\
&\le 2\left(\int_{\xT} \frac{h_{xx}^2}{(1+h_x^2)^2}\dx\right)^\mez \lA \varphi_x\rA_{L^2}=
2\left(\int_{\xT} \theta_x^2\dx\right)^\mez \lA \varphi_x\rA_{L^2},
\end{align*}
and the result follows.
\end{proof}

Our final result extends \e{eq:Ghzetafromzetax} and \e{eq:zetaxfromGhzeta} to the $L^p$-setting. 
In dimension $d=1$, 
the normal and tangential unit vectors are 
defined by 
\begin{align}\label{eq:normalnew}
N = \frac{1}{\sqrt{1+h_x^2}}\begin{pmatrix} -h_x\\ 1\end{pmatrix}\quad,\quad T = \frac{1}{\sqrt{1+h_x^2}}\begin{pmatrix}  1\\ h_x\end{pmatrix},
\end{align}
and the arc length measure on $\partial \Omega$ is $\dsigma = \sqrt{1 + h_x^2} \dx$.

\begin{thm}\label{thm:Riemest}
For all $1< p \leq 2$, there exists a constant $C_p > 0$ such that, for all 
$h\in C^{1}(\xT)$ and for all $\zeta\in H^1(\xT)$, if $\phi$ is 
defined by \eqref{eq:phi}, then the following two inequalities hold: 
\begin{equation}\label{w1}
\int_{\partial \Omega} \frac{\abs{\partial_N \phi}^p }{(1 + h_x^2)^{\frac{p-1}{2}}} \dsigma \le C_p \int_{\partial \Omega}  \abs{\partial_T \phi}^p (1 + h_x^2)^{\frac{p-1}{2}} \dsigma,
\end{equation}
and 
\begin{equation}\label{w2}
\int_{\partial \Omega} \frac{\abs{\partial_T \phi}^p }{(1 + h_x^2)^{\frac{p-1}{2}}} \dsigma \le C_p \int_{\partial \Omega}  \abs{\partial_N \phi}^p (1 + h_x^2)^{\frac{p-1}{2}} \dsigma.
\end{equation}
\end{thm}

\begin{rmk}
The estimates do not extend to $p=1$, as can be seen by assuming that $h=0$. Indeed, 
if $h=0$ and $p=1$, 
then
$$
\int_{\partial \Omega} \frac{\abs{\partial_N \phi}^p }{(1 + h_x^2)^{\frac{p-1}{2}}} \dsigma=
\int_{\xT}\la \Hil\partial_x \zeta\ra \dx\quad, \quad
\int_{\partial \Omega}  \abs{\partial_T \phi}^p (1 + h_x^2)^{\frac{p-1}{2}} \dsigma=\int_{\xT}\la \partial_x \zeta\ra \dx,
$$
where $\Hil$ is the periodic Hilbert transform (see~\eqref{eq:Hilbddnew}) 
and hence we see that the estimates do not hold for $p=1$, since $\Hil$ is not bounded on $L^1(\xT)$. 
\end{rmk}

\section{Refined Rellich estimates}\label{A:Rellich} 

In this section we prove Theorem~\ref{G-Lip-intro} and Theorem~\ref{thm:Rellich1D}. 

\subsection{Proof of Theorem~\ref{G-Lip-intro}}
The proof is decomposed into four steps. 
We start by proving the quantitative estimates~\e{d10-bisz-intro} and \e{d10-bise-intro} under the additional assumption that the functions $h$ and $\zeta$ are smooth, so that all calculations will be easily justified. Then, 
we will consider in the fourth step the general case by an approximation argument.

\bigbreak

{\em Step 1: Reduction to an estimate for $G(h)$.} 
Assume that $h$ and $\zeta$ belong to $C^\infty(\xT^d)$. 
Then \e{eq:phi} is a classical elliptic boundary problem, which admits a unique solution $\phi\in C^\infty(\overline{\Omega})$ such that 
$\nabla_{x,y}\phi\in L^2(\Omega)$. 

By definition of the Dirichlet-to-Neumann operator $G(h)$, there holds
\be\label{defi:G11}
G(h)\zeta=\big(\partial_y\phi-\nabla h\cdot\nabla \phi\big)\big\arrowvert_{y=h}=\sqrt{1+|\nabla h|^2} \partial_N \phi\big\arrowvert_{y=h}.
\ee
(Let us recall that $\nabla$ denotes the gradient with respect to  $x\in \xT^d$.)  We see that \e{d10-bisz-intro} is equivalent to
\be\label{d10-bisd}
\int_{\xT^d} \frac{(G(h)\zeta)^2}{1+|\nabla h|^2}\dx \le   
40\int_{\xT^d}(1+|\nabla h|^2)^2 |\nabla \zeta|^2 \dx.
\ee
Let us show that  \e{d10-bise-intro} also follows from \e{d10-bisd}. 
To do so, it is convenient to introduce the notations
$$
\mathcal{V}=(\nabla\phi)\arrowvert_{y=h}, \qquad \mathcal{B}=(\partial_y\phi)\arrowvert_{y=h}.
$$

Using \e{defi:G11}, we have\be\label{defi:GBV}
G(h)\zeta=\mathcal{B}-\nabla h\cdot \mathcal{V}.
\ee
On the other hand, it follows from the chain rule that
$$
\nabla \zeta=\nabla (\phi\arrowvert_{y=h})=\mathcal{V}+
\mathcal{B}\nabla h.
$$
By combining the previous identities, we see that $\mathcal{B}$ and $\mathcal{V}$ 
can be defined only in terms of $h$ and $\zeta$ by 
means of the formulas
\be\label{defi:BV2}
\mathcal{B}= \frac{G(h)\zeta+\nabla \zeta \cdot \nabla h}{1+|\nabla h|^2},\qquad 
\mathcal{V}=\nabla \zeta-\mathcal{B}\nabla h.
\ee
It follows that
\be\label{Rellich:end}
\begin{aligned}
\la (\nabla_{x,y}\phi)\arrowvert_{y=h}\ra^2&=((\partial_y\phi)\arrowvert_{y=h})^2 +
\big\vert (\nabla\phi) \arrowvert_{y=h}\big\vert^2\\
&=\mathcal{B}^2+\la \mathcal{V}\ra^2  \\
& =  \frac{(G(h)\zeta)^2}{1+|\nabla h|^2} + \abs{\grad \zeta}^2 - \frac{(\nabla h\cdot \nabla \zeta)^2}{1+|\nabla h|^2}\cdot
\end{aligned}
\ee
This shows that \e{d10-bise-intro} will follow directly from \e{d10-bisd}.

Therefore, both estimates of the theorem will be proved if we show~\e{d10-bisd}. 

\smallbreak

{\em Step 2: An intermediate Rellich type estimate.} 

To prove \e{d10-bisd}, we begin by establishing a Rellich type estimate which allows to 
estimate the $L^2$-norm of $G(h)\zeta$ in terms of $\mathcal{V}=(\nabla\phi)_{\arrowvert y=h}$. 

\begin{prop}\label{prop-Rellich}
There holds
\be\label{d10-bisb}
\int_{\xT^d} (G(h)\zeta)^2\dx \le   
\int_{\xT^d} (1+|\nabla h|^2)|\mathcal{V}|^2 \dx.
\ee
\end{prop}
\begin{proof}
By squaring the identity~\e{defi:GBV} we get
$$
(G(h)\zeta)^2
=\mathcal{B}^2-2 \mathcal{B}\nabla h \cdot \mathcal{V} +(\nabla h\cdot \mathcal{V})^2.
$$
Since $(\nabla h\cdot \mathcal{V})^2\le \la\nabla h\ra^2 \la\mathcal{V}\ra^2$, this implies
\be\label{esti:final7}
(G(h)\zeta)^2\le \mathcal{B}^2-\la\mathcal{V}\ra^2-2\mathcal{B}\nabla h\cdot \mathcal{V} +(1+|\nabla h|^2)\abs{\mathcal{V}}^2.
\ee
So,
$$
\int_{\xT^d} (G(h)\zeta)^2\dx \le
\int_{\xT^d} (1+|\nabla h|^2)\la\mathcal{V}\ra^2 \dx+R,
$$
where
\begin{align}\label{eq:R}
R=\int_{\xT^d}\Big( \mathcal{B}^2-\la\mathcal{V}\ra^2-2\mathcal{B}\nabla h\cdot \mathcal{V}\Big)\dx.
\end{align}
We see that, to obtain~\e{d10-bisb}, it is sufficient to prove 
that $R=0$. It is interesting to observe that the latter result is a consequence of the classical Rellich identity. It can be proven 
by multiplying the equation $\Delta_{x,y}\phi=0$ by $\partial_y\phi$ and then integrating by parts. We will give an alternative proof, following~\cite{A-stab-AnnalsPDE}, which consists in 
observing that $R$ is the flux associated to a vector field. Indeed, 
$$
R=\int_{\partial\Omega} X\cdot N\dsigma
$$
where 
$X\colon \Omega\rightarrow \xR^{d+1}$ is given by
$$
X=(2(\partial_y\phi)\nabla \phi; (\partial_y\phi)^2 - |\nabla \phi|^2).
$$ 
Then the key observation is that this vector field satisfies
$\cn_{x,y} X=0$ since
$$
\partial_y \big( (\partial_y\phi)^2-|\nabla\phi|^2\big)
+2\cnx \big((\partial_y\phi)\nabla\phi\big)=
2(\partial_y\phi) \Delta_{x,y}\phi=0,
$$
as can be verified by an elementary computation. 
Now, we see that the cancellation $R=0$ comes from the Stokes' theorem. 
To rigorously justify this point, we 
truncate $\Omega$ in order to work in a smooth bounded domain. Given a parameter $\beta>0$, set
$$
\Omega_\beta=\{(x,y)\in\xT^{d}\times\xR\,;-\beta<y<h(x)\}.
$$
An application of the divergence theorem in $\Omega_\beta$ gives that
$$
0=\iint_{\Omega_\beta} \cn_{x,y}X\dydx=R+\int_{\{y=-\beta\}}X\cdot n\dsigma.
$$
Recall that the potential $\phi$ satisfies \eqref{eq:phi}
$$
\lim_{y\to-\infty}\sup_{x\in\xT^{d}}\la \nabla_{x,y}\phi(x,y)\ra=0.
$$
Therefore, 
$X$ converges to $0$ uniformly when $y$ goes to $-\infty$. 
So, by sending $\beta$ to $+\infty$, we obtain the expected result $R=0$ 
which completes the proof of the proposition.
\end{proof}

{\em Step 3: Proof of \e{d10-bisd}.}

Introduce the function $\eps\colon\xT^d\to [0,+\infty)$ defined by
$$
\eps(x)\defn\frac{1}{8(1+\la\nabla h(x)\ra^2)}\cdot
$$
Introduce also the functions
$$
\lambda(x)=1+\eps(x)\quad,\quad \Lambda(x)=1+\frac{1}{\eps(x)}\cdot
$$
Directly from the identity~\e{defi:BV2} for $\mathcal{B}$ and the elementary inequality
$$
| a+b|^2\le \lambda(x) |a|^2+\Lambda(x) |b|^2\qquad (\text{for any } (a,b,x)\in\xR^d\times\xR^d\times \xT^d), 
$$
we have the pointwise inequalities
\begin{align*}
|\nabla \zeta-\mathcal{B} \nabla h|^2&\le
  \Lambda \la \nabla \zeta\ra^2+\lambda \mathcal{B}^2\la \nabla h\ra^2 \\
&\le  \Lambda \la \nabla \zeta\ra^2+\lambda \frac{\la \nabla h\ra^2}{(1+|\nabla h|^2)^2}(G(h)\zeta+\nabla \zeta \cdot \nabla h)^2\\
&\le  \Lambda \la \nabla \zeta\ra^2+\lambda^2 \frac{\la \nabla h\ra^2}{(1+|\nabla h|^2)^2}(G(h)\zeta)^2
+\lambda\Lambda \frac{\la \nabla h\ra^4}{(1+|\nabla h|^2)^2}\la \nabla \zeta \ra^2.
\end{align*}
Hence, it follows from \e{d10-bisb} that we have an estimate of the form:
$$
\int_{\xT^d} \gamma (G(h)\zeta)^2\dx \le   
\int_{\xT^d} \delta \la \nabla \zeta\ra^2\dx,
$$
where
\begin{align*}
\gamma\defn 1-\lambda^2 \frac{\la \nabla h\ra^2}{1+|\nabla h|^2},\qquad 
\delta \defn (1+|\nabla h|^2)\left(\Lambda+\lambda\Lambda \frac{\la \nabla h\ra^4}{(1+|\nabla h|^2)^2}\right).
\end{align*}

Then, we notice that
\begin{align*}
\delta&\le (1+|\nabla h|^2)(\Lambda+\lambda\Lambda)\le (1+|\nabla h|^2)\left(4+\frac{2}{\eps}\right)
\le 20(1+|\nabla h|^2)^2.
\end{align*}
On the other hand, we have
$$
\gamma =1-\lambda^2 \frac{\la \nabla h\ra^2}{1+|\nabla h|^2}=\frac{1-(2\eps+\eps^2)\la\nabla h\ra^2}{1+|\nabla h|^2}
\ge \mez\cdot \frac{1}{1+|\nabla h|^2},
$$
where we used the pointwise inequality $(2\eps+\eps^2)\la\nabla h\ra^2\le 3\eps |\nabla h|^2\le 1/2$. It follows that
$$
\mez \int_{\xT^d} \frac{(G(h)\zeta)^2}{1+|\nabla h|^2}\dx \le   
 \int_{\xT^d} 20(1+|\nabla h|^2)^2 |\nabla \zeta|^2 \dx.
$$
This implies the wanted result~\e{d10-bisd} and hence concludes the proof of the theorem.

\smallbreak

{\em Step 4: The general case.}
We now assume only that $h\in C^{1}(\xT^d)$ and 
$\zeta\in H^1(\xT^d)$. 

Introduce two sequences of smooth functions $\{h_n\}_{n\in\xN}$ 
and $\{\zeta_n\}_{n\in\xN}$ such that $\lA h_n-h\rA_{W^{1, \infty}}$ 
and $\lA \zeta_n-\zeta\rA_{H^1}$ converge to $0$ when $n$ goes to $+\infty$. 
Then it follows from variational arguments (see \cite[Section 3]{ABZ3}) 
that $G(h_n)\zeta_n$ converges to $G(h)\zeta$ in $H^{-1/2}(\xT^d)$. 

On the other hand, it follows from \e{d10-bisd} applied with $(h,\zeta)$ 
replaced by $(h_n,\zeta_n)$ that the sequence $\{G(h_n)\zeta_n\}_{n\in\xN}$ 
is bounded in $L^2(\xT^d)$, indeed
$$
\int_{\xT^d} (G(h_n)\zeta_n)^2\dx \le   
40\,  (1+\lA \nabla h_n\rA_{L^\infty}^2)^3
\int_{\xT^d} |\nabla \zeta_n|^2 \dx.
$$
It follows that there exists a subsequence $\{G(h_{n'})\zeta_{n'}\}$
converging weakly in $L^2(\xT^d)$. Therefore, by uniqueness 
of the limit in the space of distributions, we see that $G(h)\zeta$ 
belongs to $L^2(\xT^d)$. Given~\e{Rellich:end}, this in turn implies that 
$(\partial_N\phi)\arrowvert_{y=h}$ and $(\nabla_{x,y}\phi)\arrowvert_{y=h}$ 
are well defined and belong to $L^2(\xT^d)$. 

It remains to prove the estimates. 
Notice that 
$(G(h_n)\zeta_n)/\sqrt{1+|\nabla h_n|^2}$ converges weakly in $L^2$ 
to $G(h)\zeta/\sqrt{1+|\nabla h|^2}$. Therefore, 
the $L^2$-norm of the latter is bounded by 
$$
\liminf \big\Vert (G(h_n)\zeta_n)/\sqrt{1+|\nabla h_n|^2}\big\Vert_{L^2}.
$$
This establishes 
the estimate \e{d10-bisd}. Using again~\e{Rellich:end}, this in turn implies 
the estimate \e{d10-bise-intro} which completes the proof.

\subsection{Proof of \thmref{thm:Rellich1D}}

We will do the computations for smooth $h$ and $\zeta$. We can then extend the estimates to $h\in C^{1}(\xT)$ and $\zeta\in H^1(\xT)$ by the same logic as in the proof of \thmref{G-Lip-intro}. 

We know from the proof of \propref{prop-Rellich} that the quantity $R$ defined in \eqref{eq:R} is zero, i.e.
\begin{align*}
\int_{\xT} \brac{\Bcal^2 - \Vcal^2 - 2 h_x \Bcal\Vcal } \dx = 0.
\end{align*}
Now as we are in one dimension, the equations \eqref{defi:BV2} simplify
\begin{align*}
\Bcal & = \frac{h_x}{1 + h_x^2} \zeta_x +  \frac{1}{1 + h_x^2}G(h)\zeta, \\
\Vcal & = \frac{1}{1 + h_x^2}\zeta_x - \frac{h_x}{1 + h_x^2}G(h)\zeta.
\end{align*}
Substituting it in the above formula and simplifying we get
\begin{align}\label{eq:Rellich1D}
\int_{\xT} \cbrac{-\frac{\zeta_x^2}{1 + h_x^2} + \frac{(G(h)\zeta)^2}{1 + h_x^2} + \frac{2h_x\zeta_x G(h)\zeta }{1 + h_x^2}} \dx = 0.
\end{align}
Now using Young's inequality $ab \leq \frac{a^2}{2} + \frac{b^2}{2}$ gives
\begin{align*}
\int_{\xT}  \frac{(G(h)\zeta)^2}{1 + h_x^2}\dx & \leq \int_{\xT} \frac{\zeta_x^2}{1 + h_x^2}\dx + \frac{1}{2}\int_{\xT}  \frac{(G(h)\zeta)^2}{1 + h_x^2}\dx + \frac{1}{2}\int_{\xT} \frac{4h_x^2\zeta_x^2}{1 + h_x^2}\dx \\
& \leq \frac{1}{2}\int_{\xT}  \frac{(G(h)\zeta)^2}{1 + h_x^2}\dx + \int_{\xT} \frac{(1 + 2h_x^2)\zeta_x^2}{1 + h_x^2}\dx \\
& \leq \frac{1}{2}\int_{\xT}  \frac{(G(h)\zeta)^2}{1 + h_x^2}\dx +  2\int_{\xT}  \abs{\zeta_x}^2 \dx.
\end{align*} 
The estimate \eqref{eq:Ghzetafromzetax} now follows. The proof of \eqref{eq:zetaxfromGhzeta} follows the same logic.

\section{Riemann mapping and Rellich estimates}\label{S:3}
In this section, we prove~\thmref{thm:Riemest}. 

We will do the computations for smooth $h$ and $\zeta$. We can then extend the estimates to $h\in C^{1}(\xT)$ and $\zeta\in H^1(\xT)$ by the same logic as in the proof of \thmref{G-Lip-intro}. 

 Note that the estimate \eqref{eq:Ghzetafromzetax}, which reads
\begin{align*}
\int_{\xT} \frac{(G(h)\zeta)^2 }{1 + h_x^2} \dx \leq 4  \int_{\xT}  \zeta_x^2 \dx
\end{align*}
can be rewritten as 
\begin{align}\label{eq:Rellicharclength}
\int_{\partial \Omega} \frac{(\partial_N \phi)^2 }{(1 + h_x^2)^\half} \dsigma \leq 4 \int_{\partial \Omega}  (\partial_T \phi)^2 (1 + h_x^2)^\half \dsigma,
\end{align}
which is the wanted estimate~\e{w1} for $p=2$. 
We will deduce that \e{w1} holds for $1<p<2$ by an interpolation argument. To do so, 
we will exploit the existence of a Riemann mapping to reduce the problem to the study of harmonic functions in a half-space.

We first consider the $2\pi$ periodic version of $\Omega$ by considering the domain $\widetilde{\Omega} = \cbrac{(x,y) \in {\xR}^2 \suchthat \exists n \in \Zsp \tx{ so that } (x - 2n\pi, y) \in \Omega }$.  Let $\Pminus = \cbrac{ (x,y) \in {\xR}^2 \suchthat y<0}$ be the lower half plane and let $\Psi:\Pminus \to \widetilde{\Omega}$ be a Riemann mapping. As the boundary $\partial \widetilde{\Omega}$ is a Jordan curve, by Carathéodory's theorem the map $\Psi$ extends continuously to a homeomorphism on the boundary. Let $Z$ be the boundary value of $\Psi$ and so $\Z: {\xR} \to \partial \widetilde{\Omega}$ is a homeomorphism. We will denote the coordinates on this $\xR$ by $\al$ so we will use quantities like $\Z(\al), \partial_\al$ etc.

Now as $\Psi$ is a Riemann map from $\Pminus \to \widetilde{\Omega}$, we see that $z \mapsto \Psi(k(z-c))$ for $k>0$ and $c \in \xR$ are all the Riemann maps from $\Pminus \to \widetilde{\Omega}$. Therefore without loss of generality we may assume that $Z(0) = (0, h(0))$ and $Z(2\pi) = (2\pi, h(2\pi)) = Z(0) + 2\pi$. Now consider $\Psi_1: \Pminus \to  \widetilde{\Omega}$ given by $\Psi_1(z) = \Psi(z + 2\pi) - 2\pi$. Clearly $\Psi_1$ is a Riemann map with $\Psi_1(0) = \Psi(0)$ and so there exists $k>0$ so that $\Psi_1(z) = \Psi(z + 2\pi) - 2\pi = \Psi(k z)$. If $k \neq 1$, then we get a contradiction by plugging in $z = \frac{2\pi}{k-1}$ in this equation. Hence $\Psi_1 = \Psi$ and therefore $\Psi(z + 2\pi) = \Psi(z) + 2\pi$. 

As $\Psi$ is a Riemann map, we see that $\Psi_z \neq 0$ in $\Pminus$ and as $\Pminus$ is simply connected, we see that $\log(\Psi_z)$ is well defined if we fix the value of $\log(\Psi_z(-i))$ (the choice one makes is immaterial).  Now the smoothness of the domain $\widetilde{\Omega}$ implies that $\log(\Psi_z)$ extends continuous to $\Pminusbar$ (see Theorem 3.5 in \cite{Pom92}. The proof given there is for the unit disc but the same proof also works for the half plane). In particular this means that there exists $c_1, c_2 >0$ such that $ c_1 \leq \abs{\Zal(\al)} \leq c_2$ for all $\al \in \xR$. Now we define $g : \xR \to \xR$ by
\begin{align}\label{eq:logZal}
g = \Imag(\log(\Zal)).
\end{align}
Notice that $g$ is $2\pi$ periodic. 

As the slope of the interface is bounded, we can define $\theta(x) = \arctan(h_x(x))$, where $\theta$ is now the angle the interface makes with the $x$-axis. Hence we see that 
\begin{align*}
e^{i\theta(\Real(\Z(\al)))} = e^{ig(\al)}.
\end{align*} 
Therefore $1 + h_x(\Real(\Z(\al)))^2 = 1 + \tan(g(\al))^2$. We also note that $ \tan(g)$ is a bounded function. 

Now let $\widetilde{\phi}\colon \Pminus \to \xR$ be the pullback of $\phi$, given by 
\begin{align*}
\widetilde{\phi}(z) = \phi(\Psi(z)),
\end{align*}
with its boundary value being $\widetilde{\zeta}$, i.e. $\widetilde{\phi}(\al) = \widetilde{\zeta}(\al) = \zeta(\Z(\al))$. As $\Psi$ is conformal, we see that $\widetilde{\phi}$ is also a harmonic function and on the boundary we have
\begin{align*}
 (\partial_T\phi)(\Z(\al)) = \frac{1}{\Zalabs}(\partial_\al \widetilde{\phi})(\al) = \frac{1}{\Zalabs}(\pal \widetilde{\zeta})(\al) .
\end{align*}
If $n$ is the unit outward normal of $\Pminus$, then we also see that 
\begin{align*}
(\partial_N\phi)(\Z(\al)) = \frac{1}{\Zalabs}(\partial_n \widetilde{\phi})(\al) = \frac{1}{\Zalabs}(\Dabs \widetilde{\zeta})(\al)
\end{align*}
where $\Dabs = \sqrt{-\Delta}$. We can also see that the pullback of the measure $\dsigma$ on $\partial {\Omega}$  is the measure $\Zalabs\dalpha$ on $\xT$. Hence \eqref{eq:Rellicharclength} is equivalent to 
\begin{align}\label{eq:Rellichconformal}
\int_{\xT} \frac{\abs*[\big]{\Dabs\zetatil}^2}{\Zalabs (1 + \tan^2(g))^\half} \dalpha \leq 4   \int_{\xT} \frac{\abs*[\big]{\partial_\al\zetatil}^2(1 + \tan^2(g))^\half}{\Zalabs} \dalpha.
\end{align}
If $\Fcal(f)$ is the Fourier transform of $f$, then the periodic Hilbert transform  $\Hil : \Ltwo(\xT) \to  \Ltwo(\xT)$ is given by the relation
\begin{align}\label{eq:Hilbddnew}
\Fcal(\Hil f)(n) = -i sgn(n) \Fcal(f)(n)  \qq \tx{ for } n \in \Zsp,
\end{align}
where $sgn(n) = 1$ if $n>0$, $sgn(n) = -1$ if $n<0$ and $sgn(0) = 0$. Hence
$$
\abs*{\Dabs \zetatil} = \abs*{\Hil \partial_\alpha \zetatil}.
$$
Therefore we see that \eqref{eq:Rellichconformal} is equivalent to the statement that the map  $\Hil: L^2\brac{\xT, v \dalpha} \to L^2\brac{\xT,u \dalpha}$ is bounded, where 
the weights $u$ and $v$ are defined by
$$
u = \frac{(1 + \tan^2(g))^{-\half}}{\Zalabs}\quad\text{and}\quad 
v = \frac{(1 + \tan^2(g))^\half}{\Zalabs}\cdot
$$
Note that there exists constants $c_3, c_4 >0$ such that $c_3 \leq u,v \leq c_4$ on all of $\xT$ due to the properties of $\tan(g)$ and $\Zal$ mentioned above. 

Now we know that $\Hil: L^1\brac{\xT,  \dalpha} \to  L^{1,\infty}\brac{\xT, \dalpha}$ is bounded, where we recall that $ f \in L^{1,\infty} $ if we have $\norm[1,\infty]{f} = \sup_{t>0} t\abs{\cbrac{x \in \xT \suchthat \abs{f(x)} > t }} < \infty$ (see Corollary 3.16 in \cite{MusSch13}). Hence by real interpolation of operators with change of measures (namely, by using Theorem 2.9 from \cite{SteinWeiss58} with $T= \Hil$, $p_0 = q_0 = 1$, $p_1 = q_1 = 2$, $M = N = \xT$, $\diff \mu_0 = \diff \nu_0 = \diff \al$, $\diff \mu_1 = v\diff \al$ and $\diff \nu_1 = u\diff \al$) we see that, for all 
$1<p<2$,
\begin{align}\label{eq:HbddLp}
\Hil: L^p\brac{\xT, v^{p-1} \dalpha} \to L^p\brac{\xT,u^{p-1} \dalpha} \quad \tx{ is bounded.}
\end{align}
Therefore for $1<p<2$, there exists a constant $C_p>0$ such that
\begin{align*}
\int_{\xT} \frac{\abs*[\big]{\Dabs\zetatil}^p}{\Zalabs^{p-1} (1 + \tan^2(g))^{\frac{p-1}{2}}} \dalpha \leq C_p   \int_{\xT} \frac{\abs*[\big]{\partial_\al\zetatil}^p(1 + \tan^2(g))^{\frac{p-1}{2}}}{\Zalabs^{p-1}} \dalpha,
\end{align*}
which is equivalent to 
\begin{align*}
\int_{\partial \Omega} \frac{\abs{\partial_N \phi}^p }{(1 + h_x^2)^{\frac{p-1}{2}}} \dsigma \leq C_p \int_{\partial \Omega}  \abs{\partial_T \phi}^p (1 + h_x^2)^{\frac{p-1}{2}} \dsigma,
\end{align*}
proving the first statement. The other statement also follows directly as \eqref{eq:HbddLp} applied on the function $\Dabs \zetatil$ instead gets us
\begin{align*}
\int_{\xT} \frac{\abs*[\big]{\partial_\alpha\zetatil}^p}{\Zalabs^{p-1} (1 + \tan^2(g))^{\frac{p-1}{2}}} \dalpha \leq C_p   \int_{\xT} \frac{\abs*[\big]{\Dabs\zetatil}^p(1 + \tan^2(g))^{\frac{p-1}{2}}}{\Zalabs^{p-1}} \dalpha,
\end{align*}
which is equivalent to 
\begin{align*}
\int_{\partial \Omega} \frac{\abs{\partial_T \phi}^p }{(1 + h_x^2)^{\frac{p-1}{2}}} \dsigma \leq C_p \int_{\partial \Omega}  \abs{\partial_N \phi}^p (1 + h_x^2)^{\frac{p-1}{2}} \dsigma.
\end{align*}
This completes the proof.

\bigbreak

\noindent\textbf{Acknowledgements.} The authors deeply acknowledge Didier Bresch and David Lannes for several stimulating  discussions. This material is based upon a work started while the authors participated in a program hosted by the Mathematical Sciences Research Institute in Berkeley, California, during the Fall 2021 semester, supported by the National Science Foundation under Grant No. DMS-1928930. S.A. received funding from the European Research Council (ERC) under the European Union’s Horizon 2020 research and innovation program through the grant agreement 862342.  T.A. also acknowledges the SingFlows project (grant ANR-18-CE40-0027) of the French National Research Agency (ANR).

\bibliographystyle{amsplain}
\bibliography{MainRef.bib}

\providecommand{\bysame}{\leavevmode\hbox to3em{\hrulefill}\thinspace}
\providecommand{\MR}{\relax\ifhmode\unskip\space\fi MR }
\providecommand{\MRhref}[2]{%
  \href{http://www.ams.org/mathscinet-getitem?mr=#1}{#2}
}
\providecommand{\href}[2]{#2}
\begin{thebibliography}{10}

\bibitem{A-stab-AnnalsPDE}
Thomas Alazard, \emph{Stabilization of the water-wave equations with surface
  tension}, Ann. PDE \textbf{3} (2017), no.~2, Paper No. 17, 41.

\bibitem{ABZ3}
Thomas Alazard, Nicolas Burq, and Claude Zuily, \emph{On the {C}auchy problem
  for gravity water waves}, Invent. Math. \textbf{198} (2014), no.~1, 71--163.

\bibitem{AMS}
Thomas Alazard, Nicolas Meunier, and Didier Smets, \emph{Lyapunov functions,
  identities and the {C}auchy problem for the {H}ele-{S}haw equation}, Comm.
  Math. Phys. \textbf{377} (2020), no.~2, 1421--1459.

\bibitem{AmKa04}
Habib Ammari and Hyeonbae Kang, \emph{Reconstruction of small inhomogeneities
  from boundary measurements}, Lecture Notes in Mathematics, vol. 1846,
  Springer-Verlag, Berlin, 2004.

\bibitem{Brown1994}
Russell Brown, \emph{The mixed problem for {L}aplace's equation in a class of
  {L}ipschitz domains}, Comm. Partial Differential Equations \textbf{19}
  (1994), no.~7-8, 1217--1233.

\bibitem{CWGLS-2012}
Simon~N. Chandler-Wilde, Ivan~G. Graham, Stephen Langdon, and Euan~A. Spence,
  \emph{Numerical-asymptotic boundary integral methods in high-frequency
  acoustic scattering}, Acta Numer. \textbf{21} (2012), 89--305.

\bibitem{ChGrLaSp12}
\bysame, \emph{Numerical-asymptotic boundary integral methods in high-frequency
  acoustic scattering}, Acta Numer. \textbf{21} (2012), 89--305.

\bibitem{ChangLaraGuillenSchwab}
H\'{e}ctor~A. Chang-Lara, Nestor Guillen, and Russell~W. Schwab, \emph{Some
  free boundary problems recast as nonlocal parabolic equations}, Nonlinear
  Anal. \textbf{189} (2019), 11538, 60.

\bibitem{Dahlberg77}
Bj\"{o}rn E.~J. Dahlberg, \emph{Estimates of harmonic measure}, Arch. Rational
  Mech. Anal. \textbf{65} (1977), no.~3, 275--288.

\bibitem{Dahlberg87}
Bj\"{o}rn E.~J. Dahlberg and Carlos~E. Kenig, \emph{Hardy spaces and the
  {N}eumann problem in {$L^p$} for {L}aplace's equation in {L}ipschitz
  domains}, Ann. of Math. (2) \textbf{125} (1987), no.~3, 437--465.

\bibitem{Dong-Gancedo-Nguyen}
Hongjie Dong, Francisco Gancedo, and Huy~Q. Nguyen, \emph{Global well-posedness
  for the one-phase muskat problem}, Preprint (2020), arXiv:2103.02656.

\bibitem{Gao91}
Wen~Jie Gao, \emph{Layer potentials and boundary value problems for elliptic
  systems in {L}ipschitz domains}, J. Funct. Anal. \textbf{95} (1991), no.~2,
  377--399.

\bibitem{Ho54}
Lars H\"{o}rmander, \emph{Uniqueness theorems and estimates for normally
  hyperbolic partial differential equations of the second order}, Tolfte
  {S}kandinaviska {M}atematikerkongressen, {L}und, 1953, Lunds Universitets
  Matematiska Institution, Lund, 1954, pp.~105--115.

\bibitem{JeKe80}
David~S. Jerison and Carlos~E. Kenig, \emph{An identity with applications to
  harmonic measure}, Bull. Amer. Math. Soc. (N.S.) \textbf{2} (1980), no.~3,
  447--451.

\bibitem{JeKe81Dirichlet}
\bysame, \emph{The {D}irichlet problem in nonsmooth domains}, Ann. of Math. (2)
  \textbf{113} (1981), no.~2, 367--382.

\bibitem{JeKe81Neumann}
\bysame, \emph{The {N}eumann problem on {L}ipschitz domains}, Bull. Amer. Math.
  Soc. (N.S.) \textbf{4} (1981), no.~2, 203--207.

\bibitem{Lions1988}
Jacques-Louis Lions, \emph{Exact controllability, stabilization and
  perturbations for distributed systems}, SIAM Rev. \textbf{30} (1988), no.~1,
  1--68. \MR{931277 (89e:93019)}

\bibitem{McLean}
William McLean, \emph{Strongly elliptic systems and boundary integral
  equations}, Cambridge University Press, Cambridge, 2000.

\bibitem{Morawetz1968}
Cathleen~S Morawetz, \emph{Time decay for the nonlinear klein-gordon equation},
  Proceedings of the Royal Society of London A: Mathematical, Physical and
  Engineering Sciences, vol. 306, The Royal Society, 1968, pp.~291--296.

\bibitem{MusSch13}
Camil Muscalu and Wilhelm Schlag, \emph{Classical and multilinear harmonic
  analysis. {V}ol. {I}}, Cambridge Studies in Advanced Mathematics, vol. 137,
  Cambridge University Press, Cambridge, 2013.

\bibitem{Necas}
Jind\v{r}ich Ne\v{c}as, \emph{Direct methods in the theory of elliptic
  equations}, Springer Monographs in Mathematics, Springer, Heidelberg, 2012,
  Translated from the 1967 French original by Gerard Tronel and Alois Kufner,
  Editorial coordination and preface by \v{S}\'{a}rka Ne\v{c}asov\'{a} and a
  contribution by Christian G. Simader.

\bibitem{NPausader}
Huy~Q. Nguyen and Beno\^{\i}t Pausader, \emph{A paradifferential approach for
  well-posedness of the {M}uskat problem}, Arch. Ration. Mech. Anal.
  \textbf{237} (2020), no.~1, 35--100.

\bibitem{OttBrown2013}
Katharine~A. Ott and Russell~M. Brown, \emph{The mixed problem for the
  {L}aplacian in {L}ipschitz domains}, Potential Anal. \textbf{38} (2013),
  no.~4, 1333--1364.

\bibitem{PaWe55}
L.~E. Payne and H.~F. Weinberger, \emph{New bounds in harmonic and biharmonic
  problems}, J. Math. and Phys. \textbf{33} (1955), 291--307.

\bibitem{PaWe58}
\bysame, \emph{New bounds for solutions of second order elliptic partial
  differential equations}, Pacific J. Math. \textbf{8} (1958), 551--573.

\bibitem{Pom92}
Ch. Pommerenke, \emph{Boundary behaviour of conformal maps}, Grundlehren der
  mathematischen Wissenschaften [Fundamental Principles of Mathematical
  Sciences], vol. 299, Springer-Verlag, Berlin, 1992.

\bibitem{Re40}
Franz Rellich, \emph{Darstellung der {E}igenwerte von {$\Delta u+\lambda u=0$}
  durch ein {R}andintegral}, Math. Z. \textbf{46} (1940), 635--636.

\bibitem{SteinWeiss58}
Elias~M. Stein and Guido Weiss, \emph{Interpolation of operators with change of
  measures}, Trans. Amer. Math. Soc. \textbf{87} (1958), 159--172.

\bibitem{Verchota84}
Gregory Verchota, \emph{Layer potentials and regularity for the {D}irichlet
  problem for {L}aplace's equation in {L}ipschitz domains}, J. Funct. Anal.
  \textbf{59} (1984), no.~3, 572--611.

\end{thebibliography}

\vfill
\begin{flushleft}
\vspace{1cm}

\textbf{Siddhant Agrawal}\\
Instituto de Ciencias Matem\'aticas (ICMAT),\\
C/ Nicol\'as Cabrera, 13-15 (Campus Cantoblanco)\\ 
28049 Madrid \\
Spain

\vspace{1cm}

\textbf{Thomas Alazard}\\
Universit{\'e} Paris-Saclay, ENS Paris-Saclay, CNRS,\\
Centre Borelli UMR9010, avenue des Sciences, \\
F-91190 Gif-sur-Yvette\\
France

\end{flushleft}

\end{document}